\documentclass{siamltex}
\usepackage{amsfonts,latexsym}
\usepackage{epsfig}
\usepackage{amsmath}
\usepackage{graphicx}
\usepackage{amsmath}
\usepackage{color}
\usepackage{ulem}
\usepackage{algorithm,algorithmic}

\def\E{{\cal E}}
\def\I{{\cal I}}\def\L{{\cal L}}
\def\M{{\cal M}}

\def\Rs{{\mathbb{R}}}

\def\Tr{{\rm T}}

\def\beq{\begin{equation}\label}\def\eeq{\end{equation}}
\def\bea{\begin{eqnarray}\label}\def\eea{\end{eqnarray}}
\def\bas{\begin{eqnarray*}}\def\eas{\end{eqnarray*}}
\def\bt{\begin{theorem}\label}\def\et{\end{theorem}}
\def\bl{\begin{lemma}\label}\def\el{\end{lemma}}
\def\bd{\begin{definition}\label}\def\ed{\end{definition}}
\def\bex{\begin{example}\label}\def\ee{\end{example}}
\def\bc{\begin{corollary}\label}\def\ec{\end{corollary}}
\def\bp{\begin{proposition}\label}\def\ep{\end{proposition}}
\def\bpr{\begin{problem}\label}\def\epr{\end{problem}}
\def\bass{\begin{assumption}\label}\def\eass{\end{assumption}}
\def\btbl{\begin{table}[ht]\label}\def\etbl{\end{table}}

\def\la{\lambda}
\def\f{\varphi}
\def\De{\Delta}
\def\Om{\Omega}

\def\oti{\otimes}

\def\ti{\times}

\def\Inv{{\rm Inv}}

\def\Tr{^\top}\def\Inv{^{-1}}

\def\hatr{\begin{array}}\def\ear{\end{array}}
\def\bas{\begin{eqnarray*}}\def\eas{\end{eqnarray*}}
\def\bea{\begin{eqnarray}\label}\def\eea{\end{eqnarray}}
\def\beq{\begin{equation}\label}\def\eeq{\end{equation}}
\def\bde{\begin{description}}\def\ede{\end{description}}
\def\ben{\begin{enumerate}}\def\een{\end{enumerate}}
\def\bit{\begin{itemize}}\def\eit{\end{itemize}}
\def\ben{\begin{enumerate}}\def\een{\end{enumerate}}
\def\bc{\begin{center}}\def\ec{\end{center}}
\def\mat{\left[\begin{array}}\def\rix{\end{array}\right]}
\def\nn{\nonumber}

\def\Tr{^\top}

\def\beq{\begin{equation}\label}\def\eeq{\end{equation}}
\def\bea{\begin{eqnarray}\label}\def\eea{\end{eqnarray}}
\def\bas{\begin{eqnarray*}}\def\eas{\end{eqnarray*}}

\def\I2{^{-2}}

\def\*{^\ast}

\title{ Approximate solutions to large nonsymmetric differential  Riccati problems with applications to transport theory}
\author{V.  Angelova\thanks{Department of Intelligent Systems,
		Institute of Information and Communication Technologies, Bulgarian Academy of Sciences,
		Akad. G. Bonchev, bl. 2,
		Sofia 1113, Bulgaria} \and M. Hached\thanks{ Laboratoire P.  Painlev\'e UMR 8524, UFR de Math\'{e}matiques, Universit\'e des Sciences et Technologies de Lille, IUT A, Rue de la Recherche, BP 179, 59653 Villeneuve d'Ascq Cedex, France} 
\and K. Jbilou \thanks{LMPA, 50 rue F. Buisson, ULCO Calais, France}   }
\date{}
\begin{document}
\maketitle

\begin {abstract}
In the present paper, we consider large scale nonsymmetric differential matrix Riccati equations with low rank right hand sides. These matrix equations appear in many applications such as control theory,  transport theory, applied probability and others. We show how to apply  Krylov-type  methods such as the extended block Arnoldi algorithm to get low rank approximate solutions. The initial problem is projected onto small subspaces to get low dimensional nonsymmetric differential equations that are solved using the exponential approximation or via other integration schemes such as Backward Differentiation Formula (BDF) or Rosenbrok method. We also show how these technique could be easily used to solve some problems from  the well known transport equation. Some numerical experiments are given to illustrate the application of the proposed methods to large-scale problems.
\end {abstract}


\section{Introduction}

Consider the nonsymmetric differential  Riccati equation 
\begin{equation}\label{ric1}
\left\{\begin{aligned}
\dot  X(t) &=- AX(t)-X(t)D+ X(t) S X(t) + Q ,\;\; {\rm  (NDRE)}\\
X(0) &=X_0,\\
\end{aligned}
\right.
\end{equation}
where $ A \in \mathbb{R}^{n \times n} $, $ D \in \mathbb{R}^{p \times p} $, $ Q \in
\mathbb{R}^{n \times p} $, $ S \in \mathbb{R}^{p \times n} $ and $X(t)  \in \mathbb{R}^{n \times p} $ with $t \in [t_0\,,\, t_f]$.\\
The equilibrum solutions of \eqref{ric1} are the solutions of the corresponding  nonsymmetric algebraic Riccati equation
\begin{equation}\label{ric2}
- AX -XD+ X S X + Q=0.\;\; {\rm  (NARE)}\\
\end{equation}
Differential nonsymmetric  Riccati equations (NDREs) play a fundamental role in many areas such as
 transport theory, fluid queues models, 
 variational theory, optimal control and filtering, $H_1$-control, invariant embedding and scattering
processes, dynamic programming and differential games, \cite{abou,guo1,juang1,Reid,rogers}.\\
For NAREs many numerical methods have been studied for finding the minimal nonnegative solution $X^*$.   The Newton method has been  studied   in \cite{bini3,guo1,guohig}, however since  it requires at each step the solution of a Sylvester equation, the method could be expensive when direct solvers are  used.  Generally, fixed point iteration methods  \cite{abou,guo1,guohig} are less expensive than the Newton or the Schur method. Some acceleration techniques based on vector extrapolation methods \cite{jbisad} have  been proposed  in \cite{ELMoallem} to speed up the convergence of some of these fixed point iterative methods such as those introduced in \cite{lu1,lu2}.  For large problems, some Krylov-based methods have been studied in \cite{bjs}. \\For NDREs and to our  knowledge there is no existing  method in the large scale case. In this paper, we consider large scale NDREs with low rank right-hand sides. We will show how to apply the extended block Arnoldi algorithm \cite{heyouni,simoncini}  to get low rank approximate solutions. We will treat the special case corresponding to NDREs from transport theory.  \\

\noindent The paper is organized as follows: In Section 2, we will be interested in the existence of exact solutions to equation \eqref{ric1}. In Section 3, we will see how to apply the extended block Arnoldi process to get low rank approximate solutions to NDREs with  low rank right hand sides.  We give different ways for solving the obtained projected low dimensional NDREs. Some convergence and perturbation results are developed in this section. In Section 4, we investigate the BDF-Newton method for solving the problem \eqref{ric1}. Section 5 is devoted to the special case where equation \eqref{ric1} comes from transport theory. In  the last section we give  some numerical experiments.\\

\noindent Throughout this paper, we use the following notations:  The matrix $I_n$ will denote the identity matrix of size $n \times n$. The  2-norm is denoted by $\parallel .  \parallel_2$ . 

\section{Exact solutions to NDRE's}
We first  need to recall some relevant definitions
\begin{definition}\label{def1}
\begin{enumerate}
\item For any real  matrices $M=[m_{ij}]$ and $N=[n_{ij}]$ with the same size, we write $M \ge N$ if $m_{ij} \ge n_{ij}$. 
\item A real square matrix $M$ is said M-matrix if $M=sI-H$ with $H \ge 0$ and $s \ge \rho(H)$ where $\rho(.)$ denotes the spectral radius. An $M$-matrix $M$ is nonsingular if $s > \rho(H)$.
\end{enumerate}
\end{definition}

\medskip 
\noindent Let $\cal L$ be the following matrix
\begin{equation}\label{matM}
{\cal L} = \left (
\begin{array}{cc}
D & -S\\
-Q & A\\
\end{array}
\right ).
\end{equation}

\noindent In this paper, we assume that the matrix ${\cal L} $ is a nonsingular  M-matrix. It follows that the matrices $A$ and $D$ are both nonsingular M-matrices; see \cite{fguo}.\\
We notice that the special structure of the matrix ${\cal L}$  ensures  the existence of the minimal nonnegative  solution $X^*$ such that  $ X^* \ge  0$ and   $X \ge X^*$ for any solution $X$ of the NARE (\ref{ric1}),  see \cite{bini1,guo1,guohig} for more details. 

\noindent A solution of \eqref{ric2} can be expressed in the following form 
\begin{equation}
\label{sol1}
X(t)=e^{-tA} X_0 e^{-tD} + \int_0^t e^{-(t-\tau)A}Qe^{-(t-\tau)D}\, d\tau + \int_0^t e^{-(t-\tau)A}X(\tau) S X(\tau)e^{-(t-\tau)D}\, d\tau .
\end{equation}
The proof is easily done by differentiation. Now as the matrices $A$ and $D$ are also nonsingular M-matrices, they can be expressed as $A=A_1-A_2$ and $D=D_1-D_2$ where $A_2$, $D_2$ are positive matrices and $A_1$ and $A_2$ are nonsingular M-matrices. Therefore, a solution of \eqref{ric1} can be expressed as follows (see \cite{juang})
\begin{equation}
\label{sol2}
X(t)=e^{-tA_1} X_0 e^{-tD_1} + \int_0^t e^{-(t-\tau)A_1} (X(\tau) S X(\tau) +A_2X(\tau)+  X(\tau) D_2 +Q  )e^{-(t-\tau)D_1}\, d\tau.
\end{equation}
Since $\cal L$ is assumed to be a nonsingular M-matrix, then it has been proved in \cite{fguo}, by using a Picard iteration, that if $0 \le X_0 \le X^*$ where $X^*$ is a nonnegative solution of \eqref{ric2}, then there exists a global solution $X(t)$ of \eqref{ric1}.\\
It is also well known \cite{abou} that the NDRE \eqref{ric1} is related to the initial value problem
\begin{equation}\label{vp}
\left (
\begin{array}{ll}
\dot Y(t) \\
\dot Z(t) 
\end{array} 
\right ) =
\left (
\begin{array}{ll}
D & -S \\
Q & -A
\end{array} 
\right ) \; 
\left (
\begin{array}{ll}
Y(t) \\
Z(t) 
\end{array} 
\right ) ,\; Y(0)=I,\, Z(0)=X_0, 
\end{equation}
where $Y(t) \in \mathbb{R}^{p \times p}$ and $Z(t) \in \mathbb{R}^{n \times p}$. The solution of the differential linear system  \eqref{vp} is given by
\begin{equation}\label{vp2}
\left (
\begin{array}{ll}
Y(t) \\
Z(t) 
\end{array} 
\right ) = e^{t {\cal H}} \,
\left (
 \begin{array}{ll}
I\\
X_0 
\end{array} 
\right ),
\end{equation}
where $$ {\cal H}=\left (
\begin{array}{ll}
D & -S \\
Q & -A
\end{array} 
\right ).$$
Therefore, using the Radon's lemma (see \cite{abou}), we can state the following result \cite{fguo}
\begin{theorem}
\label{radon}
The problem \eqref{ric1} is equivalent to solving the linear system of differential equations \eqref{vp}.
 If the solution $X(t)$ exists on  $[0,\, \infty [$ then the solution  $Y(t)$ obtained from the problem \eqref{vp} is nonsingular and in this case  $$X(t)=Z(t)Y^{-1}(t).$$  
\end{theorem}
Using this theorem, we obtain the following result \cite{fguo}
\begin{theorem}
\label{fg}
Assume that $\cal L$ is a nonsingular $M$ matrix. If $0 \le X_0 \le X^*$ where $X^*$ is the minimal nonnegative solution of \eqref{ric2}, then the solution  $X(t)$ of \eqref{ric1}  converges to $X^*$ as $t \longrightarrow \infty$.
\end{theorem}

\section{Low rank approximate solutions to large NDREs via projection}
\subsection{The approximate solutions}
From now on, we assume that the constant matrix term $Q$ in \eqref{ric1} has a low rank and is decomposed as $Q=FG^T$ and $X_0=Z_{0,1} Z_{0,2}^T$ where $F, Z_{0,1}  \in \mathbb{R}^{n \times s}$ and $G, Z_{0,2} \in \mathbb{R}^{p \times s}$  with $ s \ll n$.  The  approach that we will consider in this section, consists in projecting the problem \eqref{ric1} onto a suitable subspace, solve the obtained low order problem and then get an approximate solution to the original problem. \\

\noindent We first recall the extended block Arnoldi process applied to the
pair $ (A,V) $ where $ A \in \mathbb{R}^{n \times n} $ is assumed to be nonsingular, and $ V \in \mathbb{R}^{n \times s} $ with $s \ll  n$. The projection subspace ${\cal K}_m(A,V)  \subset  \mathbb{R}^n $ that we will consider was
introduced in \cite{druskin,simoncini} and applied for solving large scale  symmetric differential and algebraic matrix Riccati equations in \cite{ghachedjbilou,heyouni} and for solving large scale Lyapunov matrix equations in \cite{simoncini}. This extended block Krylov subspace is given as
\begin{equation*}
 {\cal K}_m(A,V) = Range([A^{-m}V,\ldots,A^{-2}\,V,A^{-1}\,V,V,A\,V, A^2\,V, \ldots,A^{m-1}\,V]).
 \end{equation*}

\noindent The Extended Block Arnoldi (EBA) algorithm allows  the computation of  an orthonormal basis of the extended Krylov subspace $ {\cal K}_m(A,V) $. This basis contains information on both $ A $ and $ A^{-1} $. Let $ m $ be some fixed integer which limits the dimension of the constructed basis. 
The obtained blocks
$V_1,V_2,\ldots,V_m$,  ($ V_ i \in \mathbb{R}^{n \times 2s} $) have their columns mutually orthogonal provided no breakdown occurs.   After $ m $ steps, the extended block Arnoldi algorithm  builds an orthonormal basis $ {\cal V}_m = \left [V_1,\ldots,V_m \right]$ of the extended block Krylov subspace $ {\cal K}_m(A,V) $.\\

\noindent Let the matrix  $ {\cal T}_m^A \in \mathbb{R}^{2ms \times 2ms} $ denotes the
restriction of the matrix $ A $ to the extended Krylov subspace $
{\cal K}_m(A,V) $, i.e., $ {\cal T}_m^A = {\cal V}_m^T\,A\,{\cal
V}_m $. It is shown in \cite{simoncini} that $ {\cal T}_m^A $ is
a block upper Hessenberg matrix with $ 2s \times 2s $ blocks and whose elements could  be obtained recursively from  EBA.
%
Let $ {\overline{\cal T}}_m^A= {\cal V}_{m+1}^T\,A\,{\cal V}_m $, and suppose that $ m $ steps of EBA  have been  run, then we have \cite{heyouni}: 
\begin{equation}\label{rela1}
A\,{\cal V}_m = {\cal V}_{m+1}\,{\overline{\cal T}}_m^A= {\cal V}_m\,{\cal T}_m^A + V_{m+1}\,T_{m+1,m}^A\,E_m^T,
\end{equation}
and
$$A^{-1}\,{\cal V}_m ={\cal V}_{m+1}\,{\overline{\cal L}}^A_m = {\cal V}_m\,{\cal L}^A_m + V_{m+1}\,L_{m+1,m}^A\,E_m^T,$$
with $\overline{\cal L}^A_m = {\cal V}_{m+1}^{T}\,A^{-1}\,{\cal V}_m$ and ${\cal L}^A_m = {\cal V}_{m}^{T}\,A^{-1}\,{\cal V}_m$, 
where $ T_{m+1,m}^A $ and $ L_{m+1,m}^A $   are the  $ (m+1,m) $-block (of size $ 2s \times 2s $) of ${\overline{\cal T}}^A_m$ and $ {\overline{\cal L}}_m^A $, respectively and $ E_m = [ O_{2s \times 2(m-1)s}, I_{2s} ]^T $ is the matrix of the last $ 2s $ columns of the $ 2ms \times 2ms $ identity matrix $ I_{2ms} $. 

\noindent  We notice that as EBA requires mat-vec products with the matrices $A$ and $A^{-1}$, so  if the matrix $A$ is singular or when solving linear systems with $A$ is expensive, then one should use the block Arnoldi algorithm that requires only mat-vec products with the matrix $A$. In that case,  the obtained blocks $V_i$'s are of dimension $n \times s$ and form an orthonormal basis of the block Krylov subspace $\mathbb{K}(A,V)=Range([V, A\,V,\ldots,A^{m-1}\,V]) $. However, the block Arnoldi process requires generally more execution times to get good approximate solutions as compared to EBA.\\

\noindent In what follows, we will use the extended block Arnoldi algorithm, but all the results are valid when using the block Arnoldi process. To get low rank approximate solutions to \eqref{ric1}, we first  apply  the Extended Block Arnoldi (EBA) algorithm (or the block Arnoldi algorithm) to the pairs $(A,F)$ and $(D,G)$ to generate two orthonormal  bases $\{V_1,\ldots,V_m\}$ and $\{W_1,\ldots,W_m\} $ of the Extended Krylov subspaces ${\cal K}_m(A,F)$  and ${\cal K}_m(D,G)$, respectively. We obtain two orthonormal  matrices  ${\mathcal V}_m=[V_1,\ldots,V_m]$ and ${\mathcal W}_m=[W_1,\ldots,W_m]$ and two block Hessenberg matrices ${\bar {\mathcal T}_m^A}  = {\cal V}_m^T\,A\,{\cal
V}_m $ and ${\bar {\mathcal T}_m^D}= {\cal W}_m^T\,D\,{\cal
W}_m$. \\

\noindent Let $X_m(t)$ be the proposed approximate solution to (\ref{ric1}) given in the low-rank form 
\begin{equation}\label{approx1}
X_m(t) = {\mathcal V}_m Y_m(t) {\mathcal W}_m^T,
\end{equation}
satisfying the Galerkin orthogonality condition
\begin{equation}
\label{galerkin}
{\mathcal V}_m^T R_m(t) {\mathcal W}_m =0,
\end{equation}
where $R_m(t)$ is the residual $ R_m(t) = \displaystyle {\dot X_m}(t)+A\,X_m(t)+X_m(t)\,D-X_m(t)\,S\,X_m(t)- FG^T $ associated to the approximation $X_m(t)$.  Then, from \eqref{approx1} and \eqref{galerkin}, we obtain the low dimensional differential Riccati equation
\begin{equation}\label{lowric}
\left\{\begin{aligned}
\displaystyle {\dot Y}_m(t) &= - {\mathcal T}_m^A\,Y_m(t)-Y_m(t)\,{\mathcal T}_m^D + Y_m(t)\,S_m\,Y_m(t)\, +F_mG_m^T,\\
Y_m(0) & =  Y_0={\cal V}_m^T X_0 {\cal W}_m.
\end{aligned}
\right.
\end{equation}
with $ { S}_m = {\mathcal W}_m^T\,S\,{\mathcal V}_m  $, $  F_m = {\mathcal V}_m^T\,F$ and   $  G_m = {\mathcal W}_m^T\,G$. As $X_0=Z_{0,1} Z_{0,2}^T$, the initial guess $Y_0$ can be ewpressed as $Y_0={\widetilde Y}_{0,1} {{\widetilde Y}_{0,2}}^T$ where ${\widetilde Y}_{0,1} ={{\cal V}_m}^T  Z_{0,1}$ and ${\widetilde Y}_{0,2} ={\cal W}_m^T  Z_{0,2}$.\\

\noindent Therefore, the obtained low dimensional nonsymmetric differential Riccati equation \eqref{lowric} will be solved by some classical integration method that we will see in subsections 3.2 -- 3.4.\\ 
\noindent In order to stop the EBA iterations, it is desirable to be able to test if $ \parallel R_m
\parallel < \epsilon $, where $\epsilon$ is some chosen tolerance, without having to compute extra matrix  products
involving the matrices $ A $ and $D$ and their inverses. The next result gives an expression of the residual norm of $ R_m(t) $ which does not require the explicit calculation of the approximate $
X_m(t) $. A factored form will be computed only when the desired accuracy is achieved. 
\begin{theorem} \label{th2}
Let $ X_m(t) = {\mathcal V}_mY_m(t){\mathcal W}_m^T $ be the approximation obtained at step $ m $ by the Extended Block Arnoldi  method where $ Y_m $ solves  the low-dimensional differential Riccati equation (\ref{lowric}).Then 
%
%
\begin{equation}
\label{result2}
\parallel R_m(t) \parallel = \displaystyle \max \{\parallel T_{m+1,m}^AE_m^T Y_m(t) \parallel, \,  \parallel Y_m(t)E_m T_{m+1,m}^D \parallel \}
\end{equation}
where $  Y_m $ is  solution of \eqref{lowric}.
\end{theorem} 
\begin{proof}
Using the fact that $Y_m$ is a solution of the low order Riccati equation (\ref{lowric}), we  get
\begin{equation}\label{res1}
R_m(t) = {\cal V}_{m+1}
 \left (
 \begin{array}{cc}
 0&Y_m(t) E_m {\cal T}_{m+1,m}^D\\
 {\cal T}_{m+1,m}^A E_m^TY_m(t) & 0
 \end{array}\right)\,{\cal W}_{m+1}^T .
 \end{equation}
 Then since  ${\cal V}_{m+1}$ and ${\cal W}_{m+1}$ are orthonormal matrices, the result follows.
\end{proof} 


\noindent Let us see now how the obtained approximation could be expressed in a factored form. As for the algebraic case \cite{ghachedjbilou,heyouni}, using the singular value decomposition of $Y_m(t)$, and neglecting the singular values that are close to zero, the approximate solution $X_m(t)={\cal V}_m Y_m(t) {\cal W}_m^T$ can be given in the following factored form
	\begin{equation*}
	\label{appXm}
	X_m(t)\approx Z_{m,1}(t)\, Z_{m,2}^T(t),
	\end{equation*} 
	where  $Z_{m,1}(t)$ and $Z_{m,2}(t)$ are small rank matrices. \\

\noindent The following result shows  that the approximation $X_m $ is an exact solution of a perturbed  differential  Riccati equation and that the error $\mathcal{E}_m(t)=X(t)-X_m(t)$ solves another nonsymmetric differential Riccati equation.

\begin{theorem}\label{th}
Let $X_m$ be the approximate solution given by \eqref{approx1}. Then we have 
\begin{eqnarray*}
 \displaystyle {\dot X}_m(t) &=& -(A-\Delta_m^A)\,X_m(t)-X_m(t)\,(D-\Delta_m^D)+X_m(t) S\,X_m(t)+FG^T, \label{er1}\\
R_m(t) & = & \Delta_m^A X_m+X_m \Delta_m^D, and\\
{\dot {\mathcal E}_m}(t)& =& -(A-X_mS)\mathcal{E}_m(t) - \mathcal{E}_m(t)(D-SX_m) + \mathcal{E}_m(t) S\mathcal{E}_m(t) -\Delta_m^A X_m-X_m \Delta_m^D.
\end{eqnarray*}

\noindent where $ \Delta_m^A=V_{m+1}T_{m+1,m}^A V_m^T $,  $ \Delta_m^D=W_mT_{m+1,m}^D W_m^T$,  $\mathcal{E}_m(t)=X(t)-X_m(t)$ and $X$ is an exact solution of \eqref{ric1}.
\end{theorem}

\medskip
\begin{proof}
\noindent  The proof can be easily obtained  from the relation \eqref{rela1} and the  expressions of the residual $R_m(t)$ and the initial equation \eqref{ric1}. 
\end{proof}

\noindent Remark that $\Vert \Delta_m^A \Vert = \Vert T_{m+1,m}^A  \Vert $ and $\Vert \Delta_m^D \Vert = \Vert T_{m+1,m}^D  \Vert $ which shows that these two quantities tend to 0 as $m$ increases since $\Vert T_{m+1,m}  \Vert $  goes to zero as $m$ increases.\\


\noindent The matrix associated to the first nonsymmetric differential equation  in Theorem \ref{th} is given by
 \begin{equation}
 \label{mer}
 {\cal L}_m= \left (
  \begin{array}{cc}
  D-\Delta^D_m & -S\\
 -FG^T&  A-\Delta^A_m
  \end{array}\right),
 \end{equation}
also expressed as 
$${\cal L}_m= \left (
  \begin{array}{cc}
  D & -S\\
 -FG^T &  A
  \end{array}\right) - \left (
    \begin{array}{cc}
    \Delta^D_m& 0\\
 0&  \Delta^A_m
    \end{array}\right) ,$$
   This shows that the matrix ${\cal L}_m$ could be considered as a perturbation of the matrix ${\cal L}$ associated to the initial problem \eqref{ric1}. 
   Notice  that when $X_m(t)$ converges to $X(t)$ as $m$ increases, $R_m(t)=\Delta_m^A X_m+X_m \Delta_m^D$ goes to zero  and then $\Vert \Delta_m^A \Vert$ and $\Vert \Delta_m^D \Vert$ tend to zero which shows that the matrix ${\cal L}_m$ converges to the matrix ${\cal L}$.

\noindent Let us come back to the NDRE equation of the error ${\E}_m(t)$ from Theorem \ref{th}
\bea{E}
\dot \E_m(t) &=& -A_c \E_m(t) - \E_m(t)D_c+ \M(t,\E_m(t)),
\eea
where for some matrix $P$ the operator $\M(t,P)$ is defined by
\bea{5}
\M(t,P) := P(t) S P(t) -\De_m^AX_m-X_m\De_m^D,
\eea
and $A_c = A - X_mS$, $D_c = D - SX_m$, $\De_m^A = V_{m+1} T_{m+1,m}^A V\Tr_m$, $\De_m^D = W_m T_{m+1,m}^D W_m\Tr$.\\

\noindent For the  error $\E_m$ from equation (\ref{E}), the following nonlocal  bound  is valid:

\begin{theorem}	\label{err1}
Let $\Phi_P(t,t_0)$  be the fundamental matrix for the  equation $\dot \eta(t) = P \eta(t)$ for some  real matrix $P$. 

\noindent Denote 
\bea{7}
\nu &=& \max\left\{ \int_0^t \|\Phi_{A_c}(t,\tau)\|\,\|\Phi_{D_c}(\tau,t)\| \, d\tau, t \in T \right\},\\
\kappa &=&  \max \left\{ \|\Phi_{A_c}(t,0)\|\,\|\Phi_{D_c}(0,t)\| : t\in T\right\}\label{8}, 
\eea
and 
\bea{10}
a_0& = &  \nu\|S\|; \quad  a_1 = \nu\|X_m\|(\|\De_m^A\| +\|\De_m^D\|) + \kappa \|\E_m(0)\|.
\eea
Then, for the spectral norm $\|\E_m\|$ of the error $\E_m = X - X_m $,  the nonlocal  bound 
\bea{14}
& & \|\E_m\| \leq \rho =  \frac{2a_1}{1+\sqrt{1-4a_0a_1}}
\eea
is valid whenever 
\bea{13}
 \delta := \{ \|\Delta^A_m\|, \|\Delta^D_m\|\}  \in \Om := \left\{ a_0 a_1 \leq 0.25 \right\} .
 \eea
 \end{theorem}

\medskip
\begin{proof}
\noindent   Define the operator $\L(P)$ 
\bea{L}
\L(P):= \int_0^t \Phi_{A_c}(t)\Phi_{A_c}^{-1}(\tau) \, P\, \Phi_{D_c}^{-1}(\tau) \Phi_{D_c}(t)d\tau
\eea
with matrix
\bas
Mat (\L) := L: = \int_0^t\left[\Phi_{D_c}\Inv(\tau)\Phi_{D_c}(t)\right]\Tr\oti \left[\Phi_{A_c}(t) \Phi_{A_c}\Inv(\tau)\right] d\tau,
\eas
and rewrite expression (\ref{E})  in  operator form
\bea{6}
\dot{\E}_m(t) &=& \Pi(\E_m) (t), 
\eea
with
\bea{pi}
\Pi(\E_m) (t) &:=& \Phi_{A_c}(t,0) \E_m(0) \Phi_{D_c}(t,0)\
- \int_0^t \Phi_{A_c}(t,\tau ) \M(\tau, \E_m(\tau)) \Phi_{D_c}(\tau,t) d\tau\nn \\
&=& \Phi_{A_c}(t,0) \E_m(0) \Phi_{D_c}(t,0)\ + \L(-\De_m^AX_m - X_m \De_m^D) + \L(\E_m S \E_m).
\eea
Using (\ref{5}) we get
\bas
\|\M(t,P)\| &\leq & \|P\|^2\|S\| + \|X_m\|(\|\De_m^A\|+\|\De_m^D\|).
\eas

\noindent The  Lyapunov majorant for the operator $\Pi(.)$ (\ref{pi}) such that 
$ \|\Pi(\E_m) (t)\| < h(\|\E_m\|\|)$ is
\bea{9}
\|\Pi(\E_m)(t)\| \leq h(\|\E_m\|\|) := a_1+ a_0 \|\E_m\|^2,
\eea
with $a_0$, $a_1$ given in (\ref{10}).

\noindent In similar way for some $P$ and $Y$ we get
\bea{11} 
\|\Pi(P)(t) - \Pi(Y)(t)\| \leq h'(r)\|P-Y\| =2a_0r\|P-Y\|,
\eea
where $r=\max\{\|P\|,\|Y\|\}$.

\noindent Assume that there exists a number $\rho>0$, such that 
\bea{12}
h(\rho) \leq \rho,\mbox{ and } h'(\rho) < 1.
\eea
Denote by $M_{\rho} $ the set of continuous matrix valued functions $P:T\longrightarrow \Rs^{n\ti p}$ and $\|P\| \leq \rho$. Then from (\ref{9}) - (\ref{12}) it follows, that the operator $\Pi(.)$ is a contraction on $M_{\rho}$ and maps this set into itself. Hence there is a  solution $\E_m (t)$ of the operator equation (\ref{6}) such that for 
\bas
& &  \delta := \{ \|\Delta^A_m\|, \|\Delta^D_m\|\}  \in \Om := \left\{ a_0 a_1 \leq 0.25 \right\} \\
& & \|\E_m\| \leq \rho :=  \frac{2a_1}{1+\sqrt{1-4a_0a_1}}.
\eas
In what follows, the theorem is proven. \end{proof}

\noindent Using the property of the logarithmic norm, the estimates (\ref{7}), (\ref{8}) of the numbers $\nu$ and $\kappa$ take the form
\bea{16}
\|\Phi_{A_c}(\tau,0)\| &\leq & \exp\left[\int_0^\tau \la (A_c(r)) d r\right]\leq \exp \left[\int_0^\tau \la_+ (A_c(r))dr \right]\\
\|\Phi_{D_c}(\tau,0)\| &\leq & \exp\left[\int_0^\tau \xi(D_c(r))dr \right] \leq \exp\left[\int_0^\tau \xi_+(D_c(r))dr\right]\label{17},
\eea
where
\bas
\la(t) &=& 0,5 \la_{\max}\left[A_c(t)+A_c(t)\Tr\right],\\
\xi(t) &=& 0,5 \xi_{\max}\left[D_c(t) + D_c(t)\Tr\right],
\eas
are the logarithmic norms of the matrices $A_c=A-X_m S$ and $D_c=D-SX_m $, respectively. And
\bas
\nu \leq \nu_1 \leq \nu_2\\
\kappa \leq \kappa_1\leq \kappa_2
\eas
with 
\bas
\nu_1 &=& \max  \left\{\int_0^t \exp \left[\int_0^r (\la(\tau)+\xi(\tau))d \tau\right] d r : t \in T\right\}\\
\nu_2 &=& \int_0^t \exp \left[ \int_0^r(\la_+(\tau)+\xi_+(\tau))d \tau\right] dr,\\
\kappa_1 &=& \exp\left[\max \left\{\int_0^t \left(\la(\tau)+\xi(\tau)\right) d\tau : t \in T\right\}\right].\\
\kappa_2 &=& \exp\left[\int_0^t\left(\la_+(\tau) + \xi_+(\tau) \right)d\tau\right],\\
\la_+(t) &=& \left\{\begin{tabular}{rl}
	$\la(t), $ & $\la(t) > 0$\\
	0, & $\la(t) \leq 0$\\
\end{tabular}\right. \quad \quad \xi_+(t) \;\; = \;\; \left\{\begin{tabular}{rl}
	$\xi(t), $&  $\xi(t) > 0$\\
	0, & $\xi(t) \leq 0$\\
\end{tabular}\right. .
\eas

\noindent In order to obtain an explicit bound for the norm of the fundamental matrix $\|\Phi_P(t)\|$   for $P(t)= A_c(t) $ or $D_c(t) $ we can use  also the  known bounds for the matrix exponential $e^{P(t)}$ based on  power series, logarithmic norm and matrix decomposition.  Some  bounds for the matrix exponential $e^{P(t)}$  are  summarized  in \cite{PetCK91}:
\bea{ex}
\|e^{P(t)} \| \leq g(t) = c_0 e^{\varrho t} \sum_{k=0}^{p-1}(\varpi t)^k/k!,
\eea
with constants $c_0$, $\varrho$,  $\varpi$ and $p$,  listed in Table \ref{tabe}.
 \begin{table}[h!!]
  \caption{Bounds for the matrix exponential $e^{P(t)}$}\label{tabe}
 \begin{center}
	\begin{tabular}{c|c|c|c|c|c}
		\hline
		& $\mbox{Power series}$ & $ \mbox{Log norm} $ & $
		\mbox{Jordan (1)} $ & $ \mbox{Jordan (2)} $ & $ \mbox{Schur} $ \\
		\hline \hline
		$ c_0 $ & 1 & 1 & $ \mbox{cond}(Y) $ & $ \mbox{cond}(Y) $ & $ 1 $\\
		$ \varrho $ & $ \|P(t)\| $ & $ \mu(P(t)) $ & $ \alpha(P(t)) $ & $
		\alpha(P(t)) + d_{\varsigma} $ & $ \alpha(P(t)) $ \\
		$ \varpi $ & 0 & 0 & 1 & 0 & $ \varpi $ \\
		$ p $ & - & - & $ m $ & - & $ l $\\
		\hline
	\end{tabular}
\end{center}
\end{table}

\medskip
\noindent Here $ \mu (P(t)) $ is the maximum eigenvalue of the matrix 
$(P(t) + P(t)\Tr)/2 $,  $ J = Y^{-1}P(t)Y $ is the Jordan canonical
form of $ P(t) $
and $ \varsigma \geq 1 $ is the dimension of the maximum block in $
J $ (the matrix $ Y $ is chosen so that the condition number
$ \mbox{cond}(Y) = \|Y\|\|Y^{-1}\| $
is minimized), $d_\varsigma = \cos\left(\frac{\pi}{\varsigma +1}\right)$, $\alpha (P(t)) $ is the spectral abscissa of $P(t)$,
i.e. the maximum real part of the eigenvalues of $ P(t) $, and
$ T = U^H P(t)U = \Lambda + {\cal N} $ is the Schur
decomposition of $ P(t) $ where $ U $ is unitary, $ \Lambda $
is diagonal and $ {\cal N} $ is strictly upper triangular
matrix (the matrix $  U $ is chosen so that the
norm of the matrix $ {\cal N } $ is minimized),
$ l = \min\{\f: {\cal N}^\f = 0\} $ is the index of
nilpotency of $ {\cal N} $, and $\varpi = \|{\cal N }\| $ .

\subsection{Solving the projected problem using the exponential-matrix  of the low  dimensional problem}
Let us see now how to solve the projected low-dimensional nonsymmetric differential Riccati equation \eqref{lowric} which  is related to the initial value problem
\begin{equation}\label{vp1}
\left (
\begin{array}{ll}
\dot Y_{1,m}(t) \\
\dot Y_{2,m}(t) 
\end{array} 
\right ) =
\left (
\begin{array}{ll}
 {\mathcal T}_m^D & -S_m \\
F_mG_m^T & - {\mathcal T}_m^A
\end{array} 
\right ) \; 
\left (
\begin{array}{ll}
Y_{1,m}(t) \\
Y_{2,m}(t) 
\end{array} 
\right ) ,\; Y_{1,m}(0)=I\;{\rm  and} \; Y_{2,m}(0)=Y_0. 
\end{equation}

\noindent Notice that if we set
\begin{equation}\label{um}
{\cal H}_m = \left (
\begin{array}{ll}
 {\mathcal T}_m^D & -S_m \\
F_mG_m^T & - {\mathcal T}_m^A
\end{array} 
\right ),\; {\cal H} = \left (
\begin{array}{ll}
 D & -S \\
FG^T & -A
\end{array} 
\right )\;  {\rm and}\; {\cal U}_m= \left (
\begin{array}{ll}
 {\mathcal W}_m & 0 \\
0&  {\mathcal V}_m
\end{array} 
\right ),
\end{equation}
we get the following relation
\begin{equation*}
{\cal H}_m  ={\cal U}_m^T \, {\cal H} \, {\cal U}_m \;\; {\rm with}\; \;{\cal U}_m^T {\cal U}_m=I.
\end{equation*}
The solution of  the projected linear differential system \eqref{vp1} is given as
\begin{equation}\label{projs1}
\left (
\begin{array}{ll}
Y_{1,m}(t) \\
Y_{2,m}(t) 
\end{array} 
\right )
= e^{t\, {\cal H}_m }\, Z_0\,\, {\rm with}\,\,\; Z_0 = \left (
\begin{array}{cc}
I\\
Y_{0}
\end{array} 
\right ).
\end{equation}

\noindent As in general $m$ is small, the solution given by \eqref{projs1} can be obtained from Pad\'e approximants implemented in Matlab as {\tt expm}. The solution $Y_m$ of the projected nonsymmetric differential Riccati equation \eqref{vp1} is then given as 
\begin{equation}
\label{solproj}
Y_m(t)=Y_{1,m}(t) \, Y^{-1}_{2,m}(t) ,
\end{equation}
provided that $Y_{2,m}(t)$ is nonsingular and then the approximate solution to the initial problem \eqref{ric1} is defined by  $X_m={\cal V}_m Y_m {\cal W}^T_m$.\\

\noindent Another way of getting approximate solutions, is to use directly an approximation of $e^{t{\cal H} } Z_0$ as it appears in \eqref{vp2}. Using the matrices  ${\cal U}_m$  and ${\cal H}_m$ given in \eqref{um}, we propose the following approximation
\begin{equation}
\label{exp1}
e^{t {\cal H}} Z_0  \approx {\cal U}_m \, e^{t {\cal H}_m}\, \Gamma_m, \;\; {\rm with}\;\; \Gamma_m={\cal U} ^T_m Z_0.
\end{equation}
Therefore, setting 
$$\left (
\begin{array}{ll}
X_{1,m}(t) \\
X_{2,m}(t) 
\end{array} 
\right ) = {\cal U}_m\, e^{t {\cal H}_m}\, \Gamma_m,,$$
the approximate solution of the solution $X$ of \eqref{ric1} is given as 
$${\widetilde X}_m = X_{1,m}(t) {X^{-1}_{2,m}(t)}.$$

Instead of solving the low dimensional nonsymmetric differential Riccati equation \eqref{lowric} by using the exponential scheme \eqref{projs1}, we can use an integration scheme for solving ordinary differential equations such as Rosenbrock \cite{rosenbrock} or Backward Differentiation Formula (BDF)  methods \cite{asher,dieci}. That is the subject of the following two subsections.\\

 \subsection{Using the BDF integration scheme}

 At each time-step $t_k$, the approximate $Y_{m,k}$ of the $Y_m(t_k)$, where $Y_m$ is the solution to (\ref{lowric})  is then computed solving a nonsymmetric algebraic Riccati equation (NARE).  We consider the problem \eqref{lowric} and apply the $s$-step BDF method. At each iteration $k+1$ of the BDF method, the approximation $Y_{m,k+1}$ of  $Y_m(t_{k+1})$ is given by the implicit relation 
 \begin{equation}
 \label{bdf}
 Y_{m,k+1} = \displaystyle \sum_{i=0}^{s-1} \alpha_i Y_{m,k-i} +h \beta {\mathcal F}_m(Y_{m,k+1}),
 \end{equation} 
 where $h=t_{k+1}-t_k$ is the step size, $\alpha_i$ and $\beta$ are the coefficients of the BDF method as listed  in Table \ref{tab0} and ${\mathcal F}_m(X)$ is  given by 
 $${\mathcal F}_m(Y)= -{\cal T}_m^A\,Y-Y{\cal T}_m^D+Y\,S_m\,Y+F_mG_m^T.$$
 
 \begin{table}[h!!]
  \caption{Coefficients of the $s$-step BDF method with $q \le 3$.}\label{tab0}
 \begin{center}
 \begin{tabular}{c|cccc} 
 \hline
 $s$ & $\beta$ &$\alpha_0$ & $\alpha_1$ & $\alpha_2$ \\
 \hline
 1 & 1 & 1 & &\\
 2 & 2/3 & 4/3& -1/3 &\\
 3 & 6/11 & 18/11 & -9/11 & 2/11\\
 \hline
 \end{tabular}

 \end{center}
 \end{table}
 \noindent The approximate $X_{k+1}$ solves the following matrix equation
 \begin{equation*}
 -Y_{m,k+1} +h\beta (F_mG_m^T -{\cal T}_m^A Y_{m,k+1} - Y_{k+1}  {\cal T}_m^D+  Y_{m,k+1} S_m Y_{m,k+1}) + \displaystyle \sum_{i=0}^{p-1} \alpha_i Y_{m,k-i} = 0,
 \end{equation*}
 which can be written as the following  continuous-time nonsymmetric algebraic Riccati equation
 
 \begin{equation}
 \label{ricbdf}
 \mathcal{A}_m\, Y_{m,k+1}  + \,Y_{m,k+1}\, \mathcal{D}_m  -Y_{m,k+1}\, \mathcal{S}_m\,Y_{m,k+1} - \mathcal{L}_{k+1} \mathcal{G}_{k+1}^T=0,
 \end{equation}
 Where, assuming that at each timestep, $Y_{m,k}$ can be approximated as a product of  low rank factors  $Y_{m,k}\approx Z_{m,k} {\widetilde Z}_{m,k}^T$ . The coefficients matrices are given by: 
 \begin{equation*}
 \mathcal{A}_m= \frac{1}{2}I+h\beta {\cal T}_m^A,~~ \mathcal{D}_m=\frac{1}{2}I+ h\beta {\cal T}_m^D, \, \mathcal{S}_m=h\beta S_m,
 \end{equation*}
  \begin{equation*}
 \mathcal{L}_{k+1,m}=[h\beta\, F_m, \,\alpha_0 \,Z_{m,k}, \, \alpha_1\, Z_{m,k-1},\ldots,\,\ \alpha_{q-1} \, Z_{m,k-p+1}],
  \end{equation*}
  and
  \begin{equation*}
  \mathcal{G}_{k+1,m}=[G_m, \,{\widetilde Z}_{m,k} \, {\widetilde Z}_{m,k-1},\, \ldots,{\widetilde Z}_{m,k-p+1}].
  \end{equation*}
We assume that at each step $k+1$, equation \eqref{ricbdf} has a solution. 

 \subsection{Solving the low dimensional problem with the Rosenbrock method}
 Applying the two-stage Rosenbrock method \cite{butcher,rosenbrock}  to the low dimensional nonsymmetric differential Riccati  equation \eqref{lowric},  the new approximation $Y_{m,k+1}$ of  $Y_m(t_{k+1})$ obtained at step $k+1$  is defined by the relations, (see \cite{benner} for more details)
 \begin{equation}\label{ros1}
 Y_{m,k+1} =Y_{m,k}+ \displaystyle \frac{3}{2}H_1+ \frac{1}{2}H_2,
 \end{equation}
 where $H_1$ and $H_2$ solve the following Sylvester equations
 \begin{equation}\label{ros2}
 \widetilde {\mathbb{T}}^A_{m}H_1+H_1\widetilde {\mathbb{T}}^D_{m}= -\mathcal{F}(Y_{m,k}),
 \end{equation}
 \begin{equation}\label{ros3}
 \widetilde {\mathbb{T}}^A_{m}H_2+H_2\widetilde {\mathbb{T}}^D_{m}= -\mathcal{F}(Y_{m,k}+H_1)+ \displaystyle \frac{2}{h}H_1,
 \end{equation}
 where
 $$\widetilde {\mathbb{T}}^A_{m}= \gamma {\mathcal{T}}^D_{m}-\displaystyle \frac{1}{2h}I\;\;{\rm and} \;\;\widetilde {\mathbb{T}}^D_{m} = \gamma  {\mathcal{T}}^D_{m}-\displaystyle \frac{1}{2h}I,$$
 and
 $${\mathcal F}(Y)= -{\cal T}_m^A\,Y-Y{\cal T}_m^D+Y\,S_m\,Y+F_mG_m^T.$$
 The Sylvester matrix equations \eqref{ros2} and \eqref{ros3}  could be solved, for small to medium problems, by direct methods such as the Bartels-Stewart  algorithm \cite{bartels}.
 
%
%
 
  \noindent The different steps of the extended block Arnoldi algorithm for solving NDREs  are summarized in the following algorithm
 \begin{algorithm}[H]
 	\caption{[The extended block Arnoldi algorithm for NDRE's (EBA-NDRE)]}
 	\label{alg1}
 	\begin{itemize}
 		\item {\bf Inputs.}  Matrices $A$, $D$, $S$, $F$, $G$ and an integer $ m $.
 		\item {\bf Outputs} : The approximate solution in a factored form: $	X_m(t)\approx Z_{m,1}(t)\, Z_{m,2}^T(t)$.
 		\item  ~Compute the QR decompositions of $ [F,A^{-1}F] = V_1\Lambda_1 $ and $ [G,D^{-1}G] = W_1\Lambda_2 $. 
 		\item Apply the extended block Arnoldi to the pair $(A,F)$:
 		\begin{itemize}
 			\item For $ j = 1,\ldots,m $
 			\item  Set $ V_j^{(1)} $: first $ s $ columns of $ V_j $; $ V_j^{(2)} $: second $ s $ columns of $ V_j $
 			\item $ {\cal V}_j = \left [ {\cal V}_{j-1}, V_j \right ] $; $ \hat V_{j+1} = \left [ A\,V_j^{(1)},A^{-1}\,V_j^{(2)} \right ] $.
 			\item  Orthogonalize $ \hat V_{j+1} $ w.r. to $ {\cal V}_j $ to get $ V_{j+1} $, i.e.,
 			\begin{itemize}
 				\item for $ i=1,2,\ldots,j $
 				\item $ H^A_{i,j} = V_i^T\,\hat V_{j+1} $,
 				\item $ \hat V_{j+1} = \hat V_{j+1} - V_i\,H^A_{i,j} $,
 				\item endfor
 			\end{itemize}
 			\item  Compute the QR decomposition of $ \hat V_{j+1} $, i.e., $ \hat V_{j+1} = V_{j+1}\,H^A_{j+1,j} $.
 			\item endFor.
 		\end{itemize}
 		\item Apply also the extended Arnoldi process to the pair $(D,G)$ to get the blocks $W_1,\ldots,W_{m+1}$ and the upper Hessenberg matrix whose elements are $ H^D_{i,j} $.
 		\item Solve the projected NDRE \eqref{lowric} to get $Y_m(t)$ using the exponential technique, BDF or Rosenbrock method..
 		\item The approximate solution $X_m(t)$ is given by the expression \eqref{appXm}.
 	\end{itemize}
 \end{algorithm}
 
 \section{The BDF-Newton method}
 In this section, we apply directly the BDF integration scheme to the initial problem \eqref{ric1}.  Then, each time-step $t_k$, the approximate $X_{k}$ of the $X_(t_k)$,   is then computed solving a nonsymmetric algebraic Riccati equation (NARE).  Applying the $s$-step BDF method,  the approximation $X_{k+1}$ of  $X_(t_{k+1})$ is given by the implicit relation 
 \begin{equation}
 \label{bdf-N}
 X_{k+1} = \displaystyle \sum_{i=0}^{s-1} \alpha_i X_{k-i} +h \beta {\mathcal F}(X_{k+1}),
 \end{equation} 
 where $h=t_{k+1}-t_k$ is the step size, $\alpha_i$ and $\beta$ are the coefficients of the BDF method as listed  in Table \ref{tab0} and ${\mathcal F}_m(X)$ is  given by 
 $${\mathcal F}(X)= -A\,X-XD+X\,S\,X+FG^T.$$
 \noindent The approximate $X_{k+1}$ solves the following matrix equation
 \begin{equation*}
 -X_{k+1} +h\beta (FG^T -A X_{k+1} - X_{k+1} D + X_{k+1} SX_{k+1}) + \displaystyle \sum_{i=0}^{s-1} \alpha_i X_{k-i} = 0,
 \end{equation*}
 which can be written as the following  continuous-time algebraic Riccati equation
 \begin{equation}
 \label{ricbdf2}
\mathcal{G} (X_{k+1})= -\mathcal{A}\, X_{k+1}  -\,X_{k+1}\,\mathcal{D}  +X_{k+1}\, \mathcal{S} \,X_{k+1} + {{\widetilde F}_{k+1}}^T {\widetilde G}_{k+1} =0,
 \end{equation}
 Where, assuming that at each timestep, $X_k$ can be approximated as a product of  low rank factors  $X_{k}\approx Z_{k,1} {Z_{k,2}}^T$, $Z_{k,i} \in \mathbb{R}^{n \times m_k}$, with $m_k \ll n,p$. 
 The coefficients matrices are given by
 $$\mathcal{A}= h\beta A +\displaystyle \frac{1}{2}I,~~ \mathcal{D}= h\beta D +\displaystyle \frac{1}{2}I,\, \mathcal{S}=h\beta S$$
 $$  {\widetilde G}_{k+1}=[\sqrt{h\beta} G,\; \sqrt{\alpha_0}Z_{k,1}^T,\ldots,\sqrt{\alpha_{s-1}} Z_{k+1-s,1}^T],$$
 and
 $$ 
 {{\widetilde F}_{k+1}}=[\sqrt{h\beta} F, \sqrt{\alpha_0}Z_{k,2}^T,\ldots,\sqrt{\alpha_{p-1}} Z_{k+1-s,2}^T]^T
 .$$
 For large-scale problems, a common strategy of solving  the nonsymmetric Algebraic Riccati equation \eqref{ricbdf2} consists in applying the  Newton method combined with an iterative method for the numerical solution of the large-scale Sylvester equations arising at each internal iteration of the Newton's algorithm. In that case, we define a sequence of approximations to $X_{k+1}$ as follows:
 \begin{itemize}
 	\item Set $X_{k+1}^0=X_k$
 	\item Build the sequence $\left(X_{k+1}^{l}\right)_{l \in \mathbb{N}}$ defined by 
 	\begin{equation}
 	\label{newt3}
 	X_{k+1}^{l+1}=X_{k+1}^{l}-D{\mathcal G}_{X_{k+1}^{l}}({\mathcal G}(X_{k+1}^{l})),
 	\end{equation}
 \end{itemize}
 \noindent where the Fr\'echet derivative $D{\mathcal G}$ of ${\mathcal G}$ at $X_{k+1}^l$ is given by
 \begin{equation}
 D{\mathcal G}_{X_{k+1}^{l}}(H)=(\mathcal{A}- \, X_{k+1}^{l}\mathcal{S}) \, H\,+\,H \,(\mathcal{D}-\mathcal{S} \,X_{k+1}^{l})
 \end{equation} 
 A straightforward calculation proves that $X_{k+1}^{l+1}$ is the solution to the Sylvester  equation
 \begin{equation}
 \label{sylvnewt}
 (\mathcal{A}- \, X_{k+1}^{l}S) \, X\,+\,X \,(\mathcal{D}-\mathcal{S} \,X_{k+1}^{l})+X_{k+1}^{l}\, \mathcal{S}\,X_{k+1}^{l}+{\widetilde F}_{k+1} {\widetilde G}_{k+1}^T\,=0.
 \end{equation}
 The main part in each  Newton iteration is to solve a large Sylvester matrix equation with a low rank right hand side.  For small to medium problems, one can use   direct methods such as the  Bartels-Stewart algorithm \cite{bartels}. For large problems, many numerical methods have been proposed; see \cite{elguen,heyouni2,jbilou1,jbilou2,simoncini}.\\
 In our computations, we used the extended block Arnoldi   algorithm for solving the large Sylvester matrix equation (\ref{sylvnewt}). The method is defined as follows: We first apply the extended block Arnoldi (or the block Arnoldi)  to the pairs $(\mathcal{A}_k,{\widetilde F}_{k+1})$ and $({\mathcal D}_k^T,{\widetilde G}_{k+1})$ where
 $${\mathcal A}_k= \mathcal{A}-X_{k+1}^{l}\, S, {\rm and}\; {\mathcal D}_k = \mathcal{D}-\mathcal{S} \,X_{k+1}^{l}$$ and obtain a low rank approximate solution to the exact solution $X_{k+1}^{l+1}$.

Since $\mathcal{A}$ and $\mathcal{D}$ are sparse, the matrices $\mathcal{A}_k$ and $\mathcal{D}_k$ are no longer sparse and then the computation of the products $\mathcal{A}_k^{-1}Y$ and $\mathcal{D}_k^{-T}Y$ becomes very expensive. A way to overcome this drawback is to use  the Sherman-Morrison-Woodbury formula given by
\begin{equation}
\label{sherman}
(L+UV^T)^{-1} \, Y= L^{-1} Y - L^{-1}U(I + V^T L^{-1}U)V^T\, L^{-1}Y,
\end{equation}
where  $L$, $U$ and $V$ are  matrices of adequate sizes. \\
Notice that, if we  use the block Arnodi method \cite{elguen} to solve the Sylvester matrix equation (\ref{sylvnewt}), then   only matrix-block vectors products are needed. 
\medskip 

\section{Applications to NDREs from transport theory}\label{trans}

 Nonsymmetric  differential  Riccati equations (\ref{ric1}) associated with M-matrices   appear  for example in neutron  transport theory; see \cite{abou,bellman,chandar}. The problem  to be solved is given as follows
 \begin{equation}\label{transp1}
 \dot X(t)= - (\Delta -eq^T)\, X - X (\Gamma - qe^T) +Xqq^T X + ee^T.
 \end{equation}

\noindent The matrices $\Delta$ and $\Gamma$  involved in the NDRE (\ref{transp1}) have the same dimension and are given by
\begin{equation}
\Delta = {\tt diag}(\delta_1,\ldots,\delta_n),\;\;\;\; \Gamma = {\tt diag}(\gamma_1,\ldots,\gamma_n),
\end{equation}
with
\begin{equation}
\delta_i = \displaystyle \frac{1}{c\omega_i(1+\alpha)}, \;\; {\rm and} \;\;\;
\gamma_i = \displaystyle \frac{1}{c\omega_i(1-\alpha)}, i=1,\ldots,n.
\end{equation}
The vectors $e$ and $q$ are given as follows
\begin{equation}
e=(1,\ldots,1)^T,\;\; q=(q_1,\ldots,q_n)^T\;\; {\rm with}\;\; q_i =\displaystyle \frac{c_i}{2\omega_i}, i=1,\ldots,n.
\end{equation}
The matrices and vectors above depend on the two parameters $c$ and $\alpha$ satisfying  
$0 < c \le 1$, $0 \le  \alpha < 1$, and on the sequences $(\omega_i)$ and $(c_i)$, $i=1, \ldots,n$, which are the nodes and weights of the Gaussian-Legendre  quadrature on $[0,\,1]$, respectively. They are such that 
$$ 0< \omega_n<\ldots < \omega_1<1, \, {\rm and } \, \displaystyle \sum_{i=1}^n c_i=1,\;\; c_i >0,\; i=1,\ldots,n.$$
The steady-state solutions of  \eqref{transp1} satisfy the following nonsymmetric algebraic Riccati equation 
\begin{equation}\label{transp2}
  - (\Delta -eq^T)\, X - X (\Gamma - qe^T) +Xqq^T X + ee^T=0. \;\;\; ({\rm NARE})
 \end{equation}
 
For existence of solutions for  NAREs \eqref{transp2}, we have the following result
.\begin{theorem}
\cite{juanglin}
If $c=1$ and $\alpha=0$, equation \eqref{transp2} has unique nonnegative solution. Otherwise, it has two nonnegative minimal and maximal  solutions, say $X_{min}$ and $X_{max}$ with $X_{max} > X_{min} >0$. The minimal solution $X_{min}$ is strictly increasing in $c$ for a fixed $\alpha$ and decreasing in $\alpha$ for fixed $c$. 
\end{theorem}

\noindent Equation \eqref{transp1} can be expressed as follows 
\begin{equation}\label{transp3}
 \dot X(t)+\Delta X +X\, \Gamma = eq^T\, X +qe^T+Xqq^T X + ee^T.
 \end{equation}
 Therefore, integrating \eqref{transp3}, we get the following expression of a solution of \eqref{transp1}. 
 \begin{equation*}
 X(t)=e^{-t\Delta} X_0 e^{-t\Gamma} +\int_0^t e^{-(t-\tau)\Delta}  \left [   ee^T+eq^T X(\tau) + X(\tau) qe^T +X(\tau) qq^TX(\tau) \right] e^{-(t-\tau)\Delta}  d\tau.
  \end{equation*}
The global existence of a solution of  equation \eqref{transp1} was invetigated in \cite{juang,Reid} and this is stated in the following theorem 
\begin{theorem}
\cite{juang}
Let $0<c  \le 1$, $0 \le \alpha <1$. Assume that $0 \le X_0 \le X_{min}$ and $ee^T-\Delta X_0-X_0 \Gamma \ge 0$. Then a global solution $X(t)$ of \eqref{transp1}  exists and  is nondecreazing in $t$ on $[0,\, \infty[$. Futhermore,  $$\displaystyle \lim_{t \longrightarrow \infty} X(t) = X_{min},$$
where $X_{min}$ is the minimal solution of the nonsymmetric algebraic Riccati equation \eqref{transp2}.
\end{theorem}

\noindent To obtain low rank approximate solutions to \eqref{transp1}, we first apply the extended Arnoldi process to the pairs $(A,e)$ and $(D,e)$ where $A=\Delta -eq^T$ and $D=\Gamma - qe^T$ to get orthonormal bases that will be used to construct the desired low rank approximation $X_m(t)= {\cal V}_m Y_m(t) {\cal W}_m^T$ where $Y_m$ solves the low dimensional differential Riccation equation \eqref{lowric}. 
We notice that when applying the above method, we use matrix vector operations of the form $A^{-1}v$ and $D^{-1}v$. As the matrices $A$ and $D$  are  the sum of diagonal matrices  and   rank one matrices, then to reduce the costs, we can compute easily these quantities by using the Sherman-Morrison-Woodbury formula given by
\begin{equation*}
A^{-1} v={ (\Delta - eq^T)}^{-1}v= \Delta^{-1}v + \displaystyle \frac{\Delta^{-1}e\, q^T\, \Delta^{-1}v}{1-q^T\, \Delta^{-1}e},
\end{equation*}
and a similar relation for $D^{-1}v$.\\

\section{Numerical examples}
The experimental tests reported in this section illustrate the methods introduced in this work. We considered the differential nonsymmetric Riccati equation applied to transport theory \eqref{transp1} on a time interval $[t_0,t_f]$, for different values of the parameters $\alpha$ and $c$, and for several sizes. The initial condition was chosen as $X_0=Z_{0,1}Z_{0,2}^T$, where $Z_{0,1}=Z_{0,2}=O_{n \times 1}$. All the experiments were performed on  an  Intel Core i7 processor  laptop equipped with 8GB of RAM. The algorithms were coded in Matlab R2014b. The three considered methods in this work are:

- The BDF-BA-Newton method which is based on the application of a BDF(s) integration scheme to the original equation which implies, at each timestep, the resolution of the algebraic nonsymmetric Riccati equation  \eqref{ricbdf}. The latter equation is then solved by the Newton method. The numerical resolution of the Sylvester equations that need to be solved at each iteration of the Newton method is done by a Block Arnoldi method, as the coefficient matrices can be singular or ill-conditioned, impeding the use of the extended block Arnoldi algorithm.   \\

- The EBA-BDF(s) and EBA-exp methods  which consist in projecting the differential problem onto an extended Arnoldi subspace and then solve the projected nonsymmetric differential Riccati equation by a BDF method (EBA-BDF(s) method) or using the exponential method  by a quadrature method as described in section 3.2  (EBA-exp). The alternative consisting in using a Rosenbrock method instead of the BDF scheme was not useful in our experiments as it did not perform better than the BDF1. The Frobenius norm of the residual at final time is then computed and while the tolerance is not met, we repeat the process increasing the dimension of the projection subspace.  The computation of the exponential form of the solution is known for being Regarding the EBA-exp method, the Davison Maki algorithm is known to be numerically unstable and we had to use the modified Davison-Maki method to overcome this drawback, see \cite{DavMaki73} for more details.\\
For the extended block Arnoldi algorithm, the  stopping criterion was $$\| R(X_m) \|_F \, /
\,  \| F\,G^T \|_F < \, 10^{-10},$$
where the norm of the residual  $\| R_m(t_f) \|$ was computed by using Theorem \ref{th2}. For the Newton-Block Arnoldi, the  iterations were  stopped when
$$\parallel X_{k+1} - X_k \parallel_F / \| X_k\|_F < 10^{-10}.$$

\noindent {\bf Example 1.} In order to confirm that the numerical methods presented in this work produce reliable approximations, we compared their outputs to the solution  $X^{direct}(t)$ computed by the direct exponential  method  as described in Section 2, \eqref{vp}. As this direct approach is not suitable for large sized problems, we set the dimension of the problem to $n=40$. The choice of the parameters values  was $c=0.5$ and $\alpha=0.5$.  In Figure 6.1, we plotted the curves of the first component $X_{11}(t)$ for  EBA-BDF1 and for the direct exponential method on the time interval $[0,10]$.

\noindent Figure 6.2 shows that the solution of the DNRE tends to the minimal nonnegative solution $X^*$ of the algebraic nonsymmetric equation \eqref{ric2} associated to \eqref{ric1} when $t$ tends to infinity. In this figure, we plotted the errors $\Vert X_{11}^{EBA-BDF1}-X_{11}^{*} \Vert$ and $\Vert X_{11}^{EBA-exp}-X_{11}^{*} \Vert$ corresponding the the first coefficients.
\newpage

\begin{figure}[H]
	\label{fig11}
	\begin{center}
		\includegraphics[width=7cm,height=5cm]{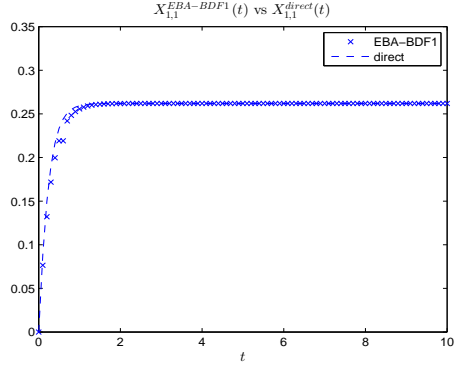}
		\caption{ First components $X_{11}(t)$, $t \in [0,10]$}
	\end{center}
\end{figure}

\begin{figure}[H]
	\label{figerr}
	\begin{center}
		\includegraphics[width=7cm,height=5cm]{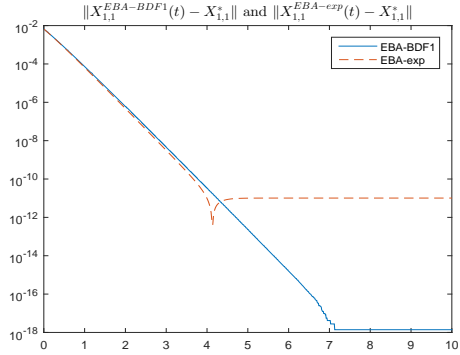}
		\caption{ Errors, corresponding to the first coefficient}
	\end{center}
\end{figure}

\noindent {\bf Example 2.}  For this example, we set  $c=0.5$ and $\alpha = 0.5$. We first computed the approximations $X_{EBA-BDF1}(t)$ , $X_{EBA-exp}(t)$ and $X_{BDF1-BA-n}(t)$ given by the EBA-BDF1, EBA-exp and BDF1-Newton-BA methods for the size $n=1000$, on the time interval $[0,1]$, for a timestep $dt=0.01$ for the BDF1 integration scheme. The relative Frobenius error norms at final time $t_f=1$ were of order $10^{-10}$ between the results of EBA-BDF1 and BDF1-BA-Newton methods whereas the EBA-exp did not performed as well with a relative error of order $10^{-4}$ when compared to both  EBA-BDF1 and BDF1-BA-Newton methods.  This problem was expected as the modified Davison-Maki requires a large number of steps in order to converge, leading to some loss of accuracy. \\
We considered problems  with the following sizes $n= 4000$,  $n = 10000$, $n = 20000$ and $n = 40000$. In Table  \ref{tab1}, we listed the obtained relative residual norms (Res.) at final time for each method and the corresponding CPU time (in seconds).
For all the experiments, the outer iterations in the Newton method did not exceed $10$ iterations. The maximum number of inner iterations was $itermax = 50$ and   were  stopped when  the corresponding residual was less than $tol = 10^{-12}$. In order to spare some computation time, the BDF1 or exponential method were performed every 5 Arnoldi iterations.

\begin{table}[h!]
	\caption{Results for the transport case $c=0.5$ and $\alpha = 0.5$.}
	\label{tab1}
	\begin{center}
	\begin{tabular}{c|c|c|c}
		\hline
		& EBA--BDF1  & EBA-Exp & BDF1-Newton-BA \\
		\hline
		\;\;$n$  & Res. \;\;\;\;\;   time & Res.  \; \;\; \;\; time  & Res.  \; \;\; \;\; time  \\
		\hline
		$4000$  &$3.9 \cdot \,10^{-9}$\; \;\;${\bf 2.9s}$ &  $4.7 \cdot 10^{-8}$\; \;\;$186s$ &  $3.9 \cdot 10^{-9}$\; \;\;$1293.4s$\\
		$10000$  &$1.1  \cdot \,10^{-8}$\; \;\;${\bf 4.4s}$ & $1.1 \cdot 10^{-8}$\; \;\;$330s$&  $--$\; \;\;$--s$  \\
		$20000$   &$2.4 \cdot \,10^{-8}$\; \;\;${\bf 7.6s}$ &  $--$\; \;\;$--s$&  $--$\; \;\;$--s$   \\
		$40000$   &$2.3\cdot \,10^{-8}$\; \,  ${\bf 12.8s}$ &$--$ \;\;  $--s$&  $--$\; \;\;$--s$  \\
		\hline
	\end{tabular}
\end{center}
	\end{table}
\noindent The results in Table \ref{tab1} show that the EBA-BDF1 method performs better than the other approaches, although all achieved satisfactory accuracies even though the EBA-exp method was not as interesting from a practical point of view. This is probably caused by the fact that the modified Davison-Maki algorithm needed a large number of sub-steps in order to converge (1000 sub-steps for the $n=4000$ case). As the number of sub-steps increases with the size of the problem, the EBA-exp could not handle the largest cases of this example.\\

\noindent {\bf Example 3.}  In this example, we repeated the tests of Example 2, for $c=0.9999$ and $\alpha=10^{-8}$. As in the previous example, the results showed a clear advantage for the methods based on the extended block Arnoldi algorithm, which are well designed for this problem. Indeed, the computations of the inverses of  the matrices $A$ and $D$ (and the forms derived from the application of the BDF integration scheme) do not require important computational efforts. \\
In Figure 6.3, we plotted the relative Frobenius residual norm of the approximate solution $X_{EBA-BDF1}(t_f)$ at final time $t_f=1$ in function of the number of extended block Arnoldi iterations for the problem size $n=4000$.

\begin{figure}[h!]
	\label{fig3}
	\begin{center}
		\includegraphics[width=7.7cm,angle=0]{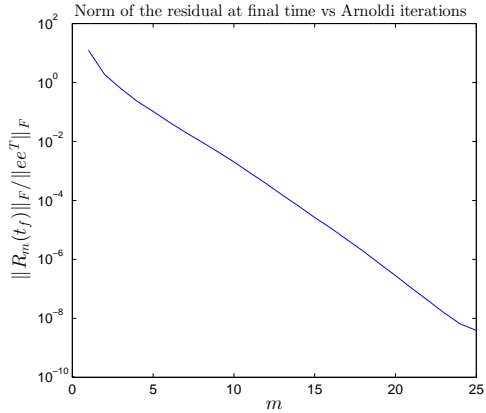}
		\caption{ Relatives Frobenius residual norms \textit{vs} the number of extended block Arnoldi iterations $m$.}\label{fig2}
	\end{center}
\end{figure}

\begin{table}[h!]
	\caption{Results for the transport case $c=0.9999$ and $\alpha = 10^{-8}$.}
	\label{tab2}
	\begin{center}
		\begin{tabular}{c|c|c|c}
			\hline
			& EBA--BDF1  & EBA-Exp & BDF1-Newton-BA \\
			\hline
			\;\;$n$  & Res. \;\;\;\;\;   time & Res.  \; \;\; \;\; time  & Res.  \; \;\; \;\; time  \\
			\hline
			$4000$  &$3.6 \cdot \,10^{-9}$\; \;\;${\bf 3.4s}$ &  $5.7 \cdot 10^{-8}$\; \;\;$183s$ &  $3.9 \cdot 10^{-9}$\; \;\;$1204.1s$\\
			$10000$  &$8.1  \cdot \,10^{-9}$\; \;\;${\bf 5.5s}$ & $4.1 \cdot 10^{-8}$\; \;\;$341s$&  $--$\; \;\;$--s$  \\
			$20000$   &$2.2 \cdot \,10^{-9}$\; \;\;${\bf 8.9s}$ &  $--$\; \;\;$--s$&  $--$\; \;\;$--s$   \\
			$40000$   &$2.3\cdot \,10^{-9}$\; \,  ${\bf 14.9s}$ &$--$ \;\;  $--s$&  $--$\; \;\;$--s$  \\
			\hline
		\end{tabular}
	\end{center}
\end{table}
The results displayed in Table \ref{tab2} confirm the good behaviour of the EBA-BDF1 method in terms of accuracy and computation time.\\

\noindent {\bf Example 4.} For this experiment, we considered the  low rank nonsymmetric differential Riccati equation (NDRE) given in 
(\ref{ric1}), for the special case (see \cite{Guo2001}) $$A=D=\begin{pmatrix} 
2&-1&&&\\
&2&\ddots&&\\
&&\ddots&&-1\\
-1&&&&2
\end{pmatrix} \;    {\rm and} \;S=diag(1,1,0,\dots,0)  \in \mathbb{R}^{n \times n}$$
The coefficients of matrices $F\in \mathbb{R}^{n\times 2}$ and $G \in \mathbb{R}^{n\times 2}$ were randomly generated. 
 In Table \ref{tab4}, we reported the obtained residual norms and the CPU times for the  EBA-BDF1  and  EBA-exp methods for various values of $n$, as the BDF-BA-Newton method is too slow to be an interesting choice in this case. In this special case, the EBA-exp method could be handled  by using the direct Davison-Maki algorithm. Both presented approaches produced equally satisfactory performances.

\begin{table}[h!]
	\caption{Results for Example 4.}
	\label{tab4}
		\begin{center}
	\begin{tabular}{l|l|l}
		\hline
		& EBA-BDF1  & EBA-exp \\
		\hline
		\;\;$n$, $p$  & Res. \;\;\;\;\;\;\;\;\;\;\; time & Res.  \;\;\;\;\;\;\;\;\;\;\;\;  time  \\
		\hline
		\hline
		
		$n=p=500$  &$7.2  \cdot \,10^{-10}$\; \;\;\;${ 0.18s}$ & $8.5 \cdot 10^{-10}$\; \;\;\;$0.08s$\\
		\hline
		
		$n=p=5000$   &$3.4 \cdot \,10^{-9}$\;\; \;\;\;${ 4.2s}$ &  $3.6 \cdot 10^{-9}$\; \;\;\;\;$3.9s$  \\
		\hline
		$n=p=10000$   &$8.6 \cdot \,10^{-9}$\;\; \;\;\;${ 20.0s}$ &  $3.9 \cdot 10^{-9}$\; \;\;\;\;$18.5s$  \\
		\hline
	\end{tabular}
	\end{center}
\end{table}

\section{Conclusion}
In this paper, we considered large-scale nonsymmetric differential Riccati equations, especially in the case arising from transport theory. We considered two approaches based on the projection of the differential equation onto an extend block Arnoldi subspace, followed by an integration scheme (BDF or exponential form via the Davison-Maki method, or its modified version). Both methods produce low rank approximates to the solution of the initial problem. We also presented an approach based on the application of the BDF scheme to the initial problem, leading to the resolution of algebraic Riccati equations which are solved by a Newton-block Arnoldi method. All three methods were able to achieve an approximate solution although the EBA-BDF1 performed better in terms of computational time. The EBA-exp method  suffered from some numerical instability which could be handled to the detriment of computational time. We reported some numerical experiments comparing those approaches for large scale problems.

\bibliographystyle{plain}
\end{document}